\documentclass[12pt]{amsart}

\usepackage{mathrsfs,amssymb}

\usepackage{tikz}
\usetikzlibrary{arrows}
\usetikzlibrary{decorations.pathreplacing}
\usepackage{float}

\newtheorem{theorem}{Theorem}[section]

\newtheorem{prop}[theorem]{Proposition}

\theoremstyle{definition}

\theoremstyle{remark}
\newtheorem{remark}[theorem]{Remark}

\numberwithin{equation}{section}

\let \la=\lambda
\let \e=\varepsilon

\let \o=\omega
\let \a=\alpha
\let \f=\varphi

\let \O=\Omega
\let \si=\sigma

\begin{document}
\title[Weak Muckenhoupt-Wheeden conjecture: sharp bound]
{On the sharp upper bound related to the weak Muckenhoupt-Wheeden conjecture}

\author{Andrei K. Lerner}
\address{Department of Mathematics,
Bar-Ilan University, 5290002 Ramat Gan, Israel}
\email{lernera@math.biu.ac.il}

\author{Fedor Nazarov}
\address{Department of Mathematics, Kent State University, OH 44240, USA}
\email{nazarov@math.kent.edu}

\author{Sheldy Ombrosi}
\address{Departamento de Matem\'atica\\
Universidad Nacional del Sur\\
Bah\'ia Blanca, 8000, Argentina}\email{sombrosi@uns.edu.ar}

\thanks{A. Lerner is supported by ISF grant No. 447/16 and ERC Starting Grant No. 713927, F. Nazarov is
supported by U.S.~National Science Foundation grant DMS-1600239, S. Ombrosi is supported by CONICET PIP 11220130100329CO, Argentina.}

\begin{abstract}
We construct an example showing that the upper bound $[w]_{A_1}\log({\rm{e}}+[w]_{A_1})$ for the $L^1(w)\to L^{1,\infty}(w)$ norm of the Hilbert transform cannot be improved in general.
\end{abstract}


\maketitle

\section{Introduction}
Define the Hardy-Littlewood maximal operator on ${\mathbb R}$ by
$$Mf(x)=\sup_{I\ni x}\frac{1}{|I|}\int_I|f(y)|dy,$$
where the supremum is taken over all intervals $I\subset {\mathbb R}$ containing the point $x$.

In \cite{FS}, C. Fefferman and E. Stein established the following weighted weak type inequality for $M$: there exists an absolute constant $C>0$ such that for every weight $w$,
\begin{equation}\label{fef-st}
\sup_{\a>0}\a w\{x\in {\mathbb R}:Mf(x)>\a\}\le C\int_{\mathbb R}|f(x)|Mw(x)dx
\end{equation}
(here by a weight we mean any non-negative locally integrable function on ${\mathbb R}$, and we use the standard notation $w(E)=\int_Ew$ for a measurable set $E\subset {\mathbb R}$).

Inequality (\ref{fef-st}) is important for several reasons. First, it is the key ingredient in extending the Hardy-Littlewood maximal theorem to a vector-valued case. Second, this result was a
precursor of the weighted theory, which had started to develop rapidly from the beginning of the 70's. In particular, if we define the $[w]_{A_1}$ constant of the weight $w$ by
$[w]_{A_1}=\|Mw/w\|_{L^{\infty}}$, then, assuming $[w]_{A_1}<\infty$, (\ref{fef-st}) implies immediately that
\begin{equation}\label{A1}
\|Mf\|_{L^{1,\infty}(w)}\le C[w]_{A_1}\|f\|_{L^1(w)}.
\end{equation}

Consider now the Hilbert transform,
$$Hf(x)=\text{P.V.}\int_{{\mathbb R}}\frac{f(y)}{x-y}dy.$$

The inequality (\ref{fef-st}) with the maximal operator on the left-hand side replaced by the Hilbert transform
has become known as the Muckenhoupt-Wheeden conjecture. Only recently this conjecture has been disproved by M. Reguera and C. Thiele \cite{RT}
(see also \cite{CLO,CS,R} for some complements and extensions). Their result, however, left open the question whether a weaker form of the
Muckenhoupt-Wheeden conjecture holds, with $M$ replaced by $H$ on the left-hand side of (\ref{A1}).

In \cite{LOP2}, it was proved that
\begin{equation}\label{lop}
\|Hf\|_{L^{1,\infty}(w)}\le C[w]_{A_1}\log({\rm{e}}+[w]_{A_1})\|f\|_{L^1(w)}.
\end{equation}
This improved a previous result in \cite{LOP1}, where the right-hand side contained an additional factor double logarithmic in $[w]_{A_1}$. Notice also that actually (\ref{lop}) in \cite{LOP2} was proved for every
Calder\'on-Zygmund operator on ${\mathbb R}^n$ with sufficiently smooth kernel.

On the other hand, in \cite{NRVV1}, it was shown for the martingale transform (and explained how to transfer the result to the Hilbert transform case) that
the dependence of $[w]_{A_1}$ in the weighted weak type $(1,1)$ inequality cannot in general be made better than $[w]_{A_1}\log^{1/5}({\rm{e}}+[w]_{A_1})$, thus disproving the weak
Muckenhoupt-Wheeden conjecture. Later, in \cite{NRVV2}, the power of the logarithm was improved to $1/3$ (this was again done for the martingale transform).

Summarizing the results in \cite{LOP2,NRVV1,NRVV2}, if we denote by $\a_{H}$ the best possible exponent for which the inequality
$$
\|Hf\|_{L^{1,\infty}(w)}\le C[w]_{A_1}\log^{\a_H}({\rm{e}}+[w]_{A_1})\|f\|_{L^1(w)}
$$
holds, then we have that $\frac{1}{3}\le \a_{H}\le 1$.

The main result of this paper shows, in particular, that $\a_{H}=1$. For $t\ge 1$, define
$$\f_H(t)=\sup_{[w]_{A_1}\le t}\|H\|_{L^1(w)\to L^{1,\infty}(w)}.$$
Then (\ref{lop}) implies $\f_H(t)\le Ct\log({\rm{e}}+t)$. We will show that actually $\f_{H}(t)\eqsim t\log({\rm{e}}+t)$. Our main result is the following.

\begin{theorem}\label{main}
There exists $c'>0$ such that for all $t\ge 1$,
$$\f_H(t)\ge c't\log({\rm{e}}+t).$$
\end{theorem}

As a trivial corollary we obtain that the inequality
$$
\|Hf\|_{L^{1,\infty}(w)}\le \psi([w]_{A_1})\|f\|_{L^1(w)}
$$
fails in general for every increasing on $[1,\infty)$ function $\psi$ satisfying $\lim_{t\to\infty}\frac{\psi(t)}{t\log({\rm{e}}+t)}=0$.

\section{Proof of Theorem \ref{main}}
\subsection{An overview of the proof}
At the first step we show that the definition of $\f_H$ along with the standard extrapolation and dualization arguments yields
\begin{equation}\label{cor}
\|H(w\chi_{[0,1)})\|_{L^2(\si)}\le 4\f_{H}(2\|M\|_{L^2(\si)\to L^2(\si)})\Big(\int_0^1w\Big)^{1/2},
\end{equation}
where $\si=w^{-1}$.
Notice that $\|M\|_{L^2(\si)\to L^2(\si)}<\infty$ if and only if $w\in A_2$, that is, if $\sup_{I\subset{\mathbb R}}\frac{w(I)\si(I)}{|I|^2}<\infty$.
Therefore, we assume here that $w\in A_2$.

The key step is to show that
there exist $C_1,C_2,C_3>0$ such that for every $t>C_3$, there is an $A_2$ weight $w$ satisfying
\begin{equation}\label{mainpr}
\int_0^1w=1,\
\|M\|_{L^2(\si)\to L^2(\si)}\le C_1t,\ \|H(w\chi_{[0,1)})\|_{L^2(\si)}\ge C_2t\log t.
\end{equation}
Plugging these estimates into (\ref{cor}), we finish the proof.

\subsection{Extrapolation and dualization} First, we apply the standard Rubio de Francia extrapolation trick.
Given $g\ge 0$ with $\|g\|_{L^2(\si)}=1$, define
$${\mathcal R}g(x)=\sum_{k=0}^{\infty}\frac{M^kg(x)}{(2\|M\|_{L^2(\si)\to L^2(\si)})^k}.$$
Then $g\le {\mathcal R}g$, $\|{\mathcal R}g\|_{L^2(\si)}\le 2$, and $[{\mathcal R}g]_{A_1}\le 2\|M\|_{L^2(\si)\to L^2(\si)}$. These estimates along with the definition of $\f_H$ and
H\"older's inequality imply
\begin{eqnarray*}
\a \int\limits_{\{x:|Hf(x)|>\a\}}g&\le& \a\int\limits_{\{x:|Hf(x)|>\a\}}{\mathcal R}g\le \f_{H}(2\|M\|_{L^2(\si)\to L^2(\si)})\|f\|_{L^1({\mathcal R}g)}\\
&\le&\f_{H}(2\|M\|_{L^2(\si)\to L^2(\si)})\|f\|_{L^2(w)}\|{\mathcal R}g\|_{L^2(\si)}\\
&\le& 2\f_{H}(2\|M\|_{L^2(\si)\to L^2(\si)})\|f\|_{L^2(w)}.
\end{eqnarray*}
Taking here the supremum over all $g\ge 0$ with $\|g\|_{L^2(\si)}=1$ yields
\begin{equation}\label{extr}
\|Hf\|_{L^{2,\infty}(w)}\le 2\f_{H}(2\|M\|_{L^2(\si)\to L^2(\si)})\|f\|_{L^2(w)}.
\end{equation}

We now use the following elementary estimate:
\begin{equation}\label{el}
\int_0^1|Hf|w\le 2\|Hf\|_{L^{2,\infty}(w)}\Big(\int_0^1w\Big)^{1/2},
\end{equation}
which along with (\ref{extr}) implies
\begin{eqnarray*}
\left|\int_{{\mathbb R}}(H(w\chi_{[0,1)}))f\right|&=&\left|\int_0^1(Hf)w\right|\\
&\le& 4\f_{H}(2\|M\|_{L^2(\si)\to L^2(\si)})\Big(\int_0^1w\Big)^{1/2}\|f\|_{L^2(w)}.
\end{eqnarray*}
Taking here the supremum over all $f$ with $\|f\|_{L^2(w)}=1$ proves (\ref{cor}).

To show (\ref{el}), notice that for every $\la>0$,
\begin{eqnarray*}
\int_0^1|Hf|w&\le& \int_{\la}^{\infty}w\{x:|Hf(x)|>\a\}d\a+\la\int_0^1w\\
&\le& \frac{1}{\la}\|Hf\|_{L^{2,\infty}(w)}^2+\la\int_0^1w.
\end{eqnarray*}
Optimizing this estimate with respect to $\la$, we obtain (\ref{el}).

\subsection{Construction of the weight} Fix $t\gg 1$. Take $k\in {\mathbb N}$ such that $t\le 3^k\le 3t$.
Let $\e=3^{-k}$ and $p=\frac{1}{3\e}\Big(\frac{1+\e}{2}+\frac{4\e^2}{1+\e}\Big)$. The reason for this definition of $p$ will be
clarified a bit later. Note that we will frequently use the obvious estimate $\frac{1}{6\e}\le p\le \frac{2}{\e}$.

For every two positive numbers $\o$ and $\si$ such that $\o\si=p$ and any interval $I\subset {\mathbb R}$, we
define inductively the sequence of weights $w_\nu=w_\nu(\o,\si,I)$ ($\nu=0,1,2,\dots$) supported on $I$ as follows.

Let $u=\sqrt p+\sqrt{p-1}$ be the larger root of $u+\frac{1}{u}=2\sqrt p$. Define
$$w_0(\o,\si,I)=\frac{\o}{\sqrt p}\Big(u\chi_{I_-}+\frac{1}{u}\chi_{I_+}\Big)\,,$$
where $I_-$ and $I_+$ are the left and the right halves of $I$ respectively.

Suppose that $w_{\nu-1}(\o,\si,I)$ is already defined for all $\o,\si$ with $\o\si=p$ and all $I\subset {\mathbb R}$.
To construct $w_\nu(\o,\si,I)$, first denote by $I_m$, $m=0,1,\dots,k-1$ the interval with the same right endpoint as $I$ of length $3^{-m}|I|$, so
$$I_{k-1}\subset I_{k-2}\subset\dots\subset I_0=I$$
and $|I_{k-1}|=3\e|I|$.

Given an interval $J$, denote by $J^{(i)}, i=1,2,3$, the $i$-th from the left subinterval of $J$ in the partition of $J$ into
$3$ equal intervals.

Define $w_\nu(\o,\si,I)$ by
\begin{eqnarray}
\,\,\,\,\,\,\,\,\,\,w_\nu(\o,\si,I)&=&\frac{\o}{p}\Biggl(\sum_{m=0}^{k-2}\chi_{I_m^{(1)}}+\chi_{I_{k-1}^{(1)}\cup I_{k-1}^{(2)}}+\frac{4\e}{1+\e}\chi_{I_{k-1}^{(3)}}\Biggr)\label{weight}\\
&+&\sum_{m=0}^{k-2}w_{\nu-1}\Big(2\o,\frac{\si}{2},I_m^{(2)}\Big).\nonumber
\end{eqnarray}

\begin{figure}[H]
\begin{center}
\begin{tikzpicture}
\draw(0,0)--(4.5,0);

\foreach \x/\xtext in {0/$$,3/$$,4.5/$$}
   \draw(\x,3pt)--(\x,-3pt) node[below] {\xtext};

\foreach \x/\xtext in {0.1/$$,0.2/$$,
0.3/$$,0.4/$$,0.5/$$,0.6/$$,0.7/$$,0.8/$$,0.9/$$,1/$$,1.1/$$,1.2/$$,1.3/$$,1.4/$$}
   \draw(\x,1pt)--(\x,-1pt) node[below] {\xtext};

\filldraw (1.5,0) circle (1pt);
\draw (0.7,0.6) node {{\textcolor{blue}{$\frac{\o}{p}$}}};
\draw[decorate, decoration={brace}, yshift=2ex]  (1.6,0) -- node[above=0.4ex] {$w_{\nu-1}$}  (2.9,0);
\draw[decorate, decoration={brace}, yshift=2ex]  (3.1,0) -- node[above=0.4ex] {$I_{m+1}$}  (4.4,0);
\draw[decorate, decoration={brace,mirror}, yshift=-8ex]  (0,1) -- node[below=0.4ex] {$I_m$}  (4.5,1);

\draw(6,0)--(9,0);

\foreach \x/\xtext in {6/$$,9/$$}
   \draw(\x,3pt)--(\x,-3pt) node[below] {\xtext};

\filldraw (7,0) circle (1pt);
\filldraw (8,0) circle (1pt);

\foreach \x/\xtext in {6.1/$$,6.2/$$,
6.3/$$,6.4/$$,6.5/$$,6.6/$$,6.7/$$,6.8/$$,6.9/$$}
   \draw(\x,1pt)--(\x,-1pt) node[below] {\xtext};

\foreach \x/\xtext in {7.1/$$,7.2/$$,
7.3/$$,7.4/$$,7.5/$$,7.6/$$,7.7/$$,7.8/$$,7.9/$$}
   \draw(\x,1pt)--(\x,-1pt) node[below] {\xtext};

\foreach \x/\xtext in {8.1/$$,8.2/$$,
8.3/$$,8.4/$$,8.5/$$,8.6/$$,8.7/$$,8.8/$$,8.9/$$}
   \draw(\x,1pt)--(\x,-1pt) node[below] {\xtext};

\draw (6.5,0.6) node {{\textcolor{blue}{$\frac{\o}{p}$}}};
\draw (7.5,0.6) node {{\textcolor{blue}{$\frac{\o}{p}$}}};
\draw (8.5,0.6) node {{\textcolor{blue}{$\frac{4\e\o}{(1+\e)p}$}}};
\draw[decorate, decoration={brace,mirror}, yshift=-8ex]  (6,1) -- node[below=0.4ex] {$I_{k-1}$}  (9,1);
\end{tikzpicture}
\caption{$w_{\nu}(\o,\si,I)$ on intervals $I_{m}^{(i)}$ for $i=1,2$ and $0\le m\le k-2$ and on $I_{k-1}^{(i)}$ for $i=1,2,3$.}
\end{center}
\end{figure}
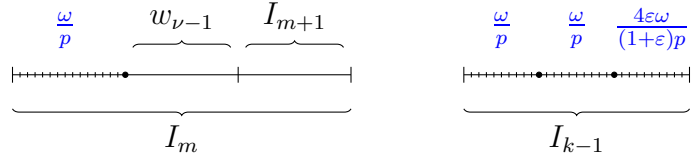

Note that the interval $I^{(3)}_{k-1}$  plays a rather special role in the final step of this recursive construction. We shall call any interval of this type
arising at any step in the construction of the weight $w_\nu(\o,\si,I)$
a ``tail interval'', so within $I$ we shall have one tail interval $I^{(3)}_{k-1}$ arising at the final stage of the construction,
$k-1$ tail intervals $(I_m^{(2)})^{(3)}_{k-1}$ arising in the construction of the weights $w_{\nu-1}(2\o,\si/2,I_m^{(2)})$, and so on.

Finally, we define $w$ as the $1$-periodization of $w_n(1,p,[0,1))$ with $n=3^{k-1}$.


For $l=0,1,\dots,n$, we shall say that an interval $I$ ``carries $w_{n-l}$'' if $w=w_{n-l}(2^l, 2^{-l}p,I)$ on~$I$. Denote by $\text{supp}\,w_{n-l}$ the union of all intervals
carrying $w_{n-l}$. For example, $\text{supp}\,w_n=\cup_{k\in {\mathbb Z}}[k,k+1)=\mathbb R$ as $[k,k+1)$
carries~$w_n$ for every $k\in{\mathbb Z}$.

Let us now establish several useful properties of $w_\nu(\o,\si,I)$.

\begin{prop}\label{prop1} For every $\nu\ge 0$,
\begin{equation}\label{av}
\frac{1}{|I|}\int_Iw_\nu(\o,\si,I)dx=\o\quad\text{and}\quad \frac{1}{|I|}\int_Iw_\nu^{-1}(\o,\si,I)dx=\si.
\end{equation}
\end{prop}

\begin{proof} The proof is by induction on $\nu$. For $\nu=0$,
$$
\frac{1}{|I|}\int_Iw_0(\o,\si,I)dx=\frac{\o}{\sqrt p}\frac{1}{2}(u+1/u)=\o,
$$
and
$$
\frac{1}{|I|}\int_Iw_0^{-1}(\o,\si,I)dx=\frac{\sqrt p}{\o}\frac{1}{2}(1/u+u)=\frac{p}{\o}=\si.
$$

Assume that the statement holds for $\nu-1$ and let us prove it for $\nu$. Observe that $w_\nu(\o,\si,I)$ equals $\frac{\o}{p}$ on a subset of $I$ of total measure
$$\frac{1-3\e}{2}|I|+2\e|I|=\frac{1+\e}{2}|I|,$$
it equals $\frac{4\e}{1+\e}\frac{\o}{p}$ on a set of measure $\e|I|$, and the average of $w_{\nu-1}(2\o,\cdot,\cdot)$ over the remaining set of measure $\frac{1-3\e}{2}|I|$ is $2\o$
by the inductive assumption. Thus
\begin{eqnarray*}
\frac{1}{|I|}\int_Iw_\nu(\o,\si,I)dx&=&\frac{\o}{p}\left(\frac{1+\e}{2}+\frac{4\e^2}{1+\e}\right)+\o(1-3\e)\\
&=&\o+\left(\frac{1}{p}\Big(\frac{1+\e}{2}+\frac{4\e^2}{1+\e}\Big)-3\e\right)\o=\o
\end{eqnarray*}
(it is this equation that was used to determine $p$).

On the other hand, $w_\nu^{-1}(\o,\si,I)$ equals $\frac{p}{\o}=\si$ on a subset of $I$ of measure $\frac{1+\e}{2}|I|$
(the same set on which $w_\nu$ is defined as $\o/p$), it equals $\frac{1+\e}{4\e}\si$ on a set of measure $\e|I|$, and its average over the remaining set
of measure $\frac{1-3\e}{2}|I|$ equals $\frac{\si}{2}$. Thus
$$
\frac{1}{|I|}\int_Iw_\nu^{-1}(\o,\si,I)dx=\frac{1+\e}{2}\si+\frac{1+\e}{4}\si+\frac{(1-3\e)}{2}\frac{\si}{2}=\si,
$$
which completes the proof.
\end{proof}

In particular, it follows from Proposition \ref{prop1} that for the constructed weight $w$,
$$
\int_0^1wdx=\int_0^1w_n(1,p,[0,1))dx=1.
$$

\begin{prop}\label{prop2} Let $I=[a,a+h)$. Then, for every $\nu\ge 0$ and for all $0<\tau<h$,
\begin{equation}\label{avest}
\frac{1}{\tau}\int_a^{a+\tau}w_\nu(\o,\si,I)\le 3\o\quad\text{and}\quad \frac{1}{\tau}\int_{a+h-\tau}^{a+h}w_\nu(\o,\si,I)\le \frac{9}{2}\o.
\end{equation}
\end{prop}

\begin{proof} For $\nu=0$ the statement is obvious since $w_0(\o,\si,I)\le 2\o$ on~$I$. Assume that $\nu\ge 1$.

Since $w_\nu(\o,\si,I)=\frac{\o}{p}$ on $I^{(1)}$, we have that
$\frac{1}{\tau}\int_a^{a+\tau}w_\nu(\o,\si,I)=\frac{\o}{p}$ for $0<\tau<h/3$. But if $\tau\ge h/3$, then, by Proposition \ref{prop1},
$$\frac{1}{\tau}\int_a^{a+\tau}w_\nu(\o,\si,I)dx\le \frac{3}{|I|}\int_Iw_\nu(\o,\si,I)dx=3\o\,.$$

We now turn to the proof of the second estimate in (\ref{avest}).
Let $I_m, m=0,1,\dots, k-1,$ be the intervals appearing in the definition of $w_\nu(\o,\si,I)$.
Since $w_\nu(\o,\si,I)\le \frac{\o}{p}$ on $I_{k-1}$, the estimate is trivial if $a+h-\tau\in I_{k-1}$.
Assume that $a+h-\tau\in I_{m}\setminus I_{m+1},m=0,\dots,k-2$. Then
\begin{equation}\label{int}
\frac{1}{\tau}\int_{a+h-\tau}^{a+h}w_\nu(\o,\si,I)dx\le \frac{1}{|I_{m+1}|}\int_{I_m}w_\nu(\o,\si,I)dx.
\end{equation}
Next, by Proposition \ref{prop1},
\begin{eqnarray}
\int_{I_m}w_\nu(\o,\si,I)dx&=&\sum_{j=m}^{k-2}\left(\frac{\o}{p}|I_j^{(1)}|+\int_{I_j^{(2)}}w_{\nu-1}(2\o,\si/2,I_j^{(2)})dx\right)\nonumber\\
&+&\int_{I_{k-1}}w_\nu(\o,\si,I)dx\nonumber\\
&\le&\o\sum_{j=m}^{k-1}|I_j|\le \frac{3\o}{2}\frac{|I|}{3^m}=\frac{9\o}{2}|I_{m+1}|\label{intest},
\end{eqnarray}
which along with (\ref{int}) completes the proof.
\end{proof}

Assume that $I$ carries $w_{n-l}$. Consider the corresponding tail intervals contained in $I$
(that is, the intervals on which $w=\frac{4\e}{1+\e}\frac{2^j}{p}, j=l,\dots,n-1$).

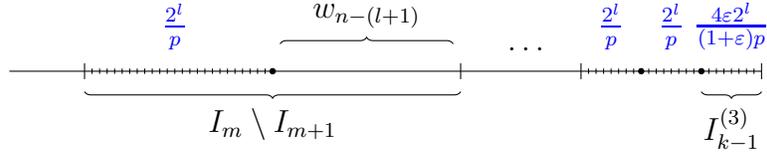
\begin{figure}[H]
\begin{center}
\begin{tikzpicture}
\draw(0,0)--(10,0);

\foreach \x/\xtext in {1/$$,6/$$,   7.6/$$,10/$$}
   \draw(\x,3pt)--(\x,-3pt) node[below] {\xtext};

\filldraw (8.4,0) circle (1pt);
\filldraw (9.2,0) circle (1pt);
\filldraw (3.5,0) circle (1pt);

\foreach \x/\xtext in {7.7/$$,7.8/$$,
7.9/$$,8/$$,8.1/$$,8.2/$$,8.3/$$}
\draw(\x,1pt)--(\x,-1pt) node[below] {\xtext};

\foreach \x/\xtext in {8.5/$$,8.6/$$,
8.7/$$,8.8/$$,8.9/$$,9/$$,9.1/$$}
\draw(\x,1pt)--(\x,-1pt) node[below] {\xtext};

\foreach \x/\xtext in {9.3/$$,9.4/$$,
9.5/$$,9.6/$$,9.7/$$,9.8/$$,9.9/$$}
\draw(\x,1pt)--(\x,-1pt) node[below] {\xtext};

\foreach \x/\xtext in {1.1/$$,1.2/$$,
1.3/$$,1.4/$$,1.5/$$,1.6/$$,1.7/$$,1.8/$$,
1.9/$$,2/$$,2.1/$$,2.2/$$,2.3/$$,2.4/$$,
2.5/$$,2.6/$$,2.7/$$,2.8/$$,2.9/$$,3/$$,
3.1/$$,3.2/$$,3.3/$$,3.4/$$}
\draw(\x,1pt)--(\x,-1pt) node[below] {\xtext};

\draw (2.2,0.6) node {{\textcolor{blue}{$\frac{2^l}{p}$}}};
\draw[decorate, decoration={brace}, yshift=2ex]  (3.6,0) -- node[above=0.4ex] {$w_{n-(l+1)}$}  (5.9,0);
\draw (8,0.6) node {{\textcolor{blue}{$\frac{2^l}{p}$}}};
\draw (8.8,0.6) node {{\textcolor{blue}{$\frac{2^l}{p}$}}};
\draw (9.6,0.6) node {{\textcolor{blue}{$\frac{4\e2^l}{(1+\e)p}$}}};
\draw (6.9,0.3) node {$\dots$};
\draw[decorate, decoration={brace,mirror}, yshift=-7ex]  (9.2,1) -- node[below=0.4ex] {$I^{(3)}_{k-1}$}  (10,1);
\draw[decorate, decoration={brace,mirror}, yshift=-7ex]  (1,1) -- node[below=0.4ex] {$I_m\setminus I_{m+1}$}  (6,1);
\end{tikzpicture}
\caption{$w$ on some interval $I$ carrying $w_{n-l}$ for $l<n$.}
\end{center}
\end{figure}

These intervals will play the central role in the estimate of the Hilbert transform of $w\chi_{[0,1)}$.
There is only one tail interval in $I\setminus \text{supp}\,w_{n-(l+1)}$, and its measure equals $\frac{1}{3^k}|I|$.
Next, there are $k-1$ tail intervals
in
$$I\cap\big(\text{supp}\,w_{n-(l+1)}\setminus \text{supp}\,w_{n-(l+2)}\big)$$
of total measure
$\frac{1}{2}\Big(1-\frac{1}{3^{k-1}}\Big)\frac{1}{3^k}|I|\,.$
Similarly, the measure of the union of tail intervals in
$$I\cap\big(\text{supp}\,w_{n-(l+j)}\setminus \text{supp}\,w_{n-(l+j+1)}\big)\quad(j=0,\dots,n-l-1)$$
is $\Big(\frac{1}{2}\Big(1-\frac{1}{3^{k-1}}\Big)\Big)^j\frac{1}{3^k}|I|$.

In particular, if we denote by $A_l$ the union of tail intervals in
$$[0,1)\cap\big(\text{supp}\,w_{n-l}\setminus \text{supp}\,w_{n-(l+1)}\big),$$
then
\begin{equation}\label{Al}
|A_l|=\Big(\frac{1}{2}\Big(1-\frac{1}{3^{k-1}}\Big)\Big)^l\frac{1}{3^k}\quad(l=0,\dots,n-1).
\end{equation}

\subsection{Estimate of the maximal operator} In this section, we will prove the first inequality in (\ref{mainpr}). We start with the reduction of this estimate to its
triadic version.

Let ${\mathcal T}$ be the standard triadic lattice, that is,
$${\mathcal T}=\{[3^jn,3^j(n+1)),\quad j,n\in {\mathbb Z}\}.$$
Denote by ${\mathcal J}$ the family of all unions of two adjacent triadic intervals of equal length.

Our key tool will be the following estimate:
\begin{equation}\label{keytool}
\|M\|_{L^2(\si)\to L^2(\si)}\le 24\sup_{J\in{\mathcal J}}\left(\frac{1}{w(J)}\int_J(M(w\chi_J))^2\si\right)^{1/2}.
\end{equation}
This estimate is fairly standard and well-known. For reader's convenience, we supply the proof in the Appendix.

Combining (\ref{keytool}) with the inequality $p\le \frac{2}{\e}\le 6t$, we see that
in order to prove the first estimate in (\ref{mainpr}), it
suffices to show that there exists $C>0$ such that for every interval $J\in {\mathcal J}$,
\begin{equation}\label{sufshow}
\int_J(M(w\chi_J))^2\si\le Cp^2w(J).
\end{equation}

Define an auxiliary $1$-periodic function $\widetilde w$ by
$$\widetilde w(x)=\sum_{l=1}^n2^l\chi_{\text{supp}\,w_{n-(l-1)}\setminus\text{supp}\,w_{n-l}}(x)+2^{n+1}\chi_{\text{supp}\,w_0}(x)\,.$$
The role of this function is clarified in the following two propositions.

\begin{prop}\label{max} For all $x\in {\mathbb R}$,
\begin{equation}\label{maxest}
Mw(x)\le \frac{9}{2}\widetilde w(x)\,.
\end{equation}
\end{prop}

\begin{proof} First, notice that for $x\in \text{supp}\,w_0$ the statement is trivial. Indeed, $w\le\frac{2^{n-1}}{p}$ on the complement of ${\rm{supp}}\,w_{0}$,
and if $I$ carries $w_0$, then on $I$
$$w_0=\frac{2^n}{\sqrt p}\Big((\sqrt{p}+\sqrt{p-1})\chi_{I_-}+\frac{1}{\sqrt{p}+\sqrt{p-1}}\chi_{I_{+}}\Big).$$
Hence,
$$\|w\|_{L^{\infty}}\le \frac{2^n(\sqrt{p}+\sqrt{p-1})}{\sqrt p}\le 2^{n+1},$$
and therefore $\|Mw\|_{L^{\infty}}\le 2^{n+1}$.

On the other hand, for $x\in \text{supp}\,w_{n-(l-1)}\setminus\text{supp}\,w_{n-l}$, the estimate (\ref{maxest})
follows immediately from the facts that $w\le \frac{2^{l-1}}{p}$ on the complement of  ${\rm{supp}}\,w_{n-l}$
and that, by Proposition \ref{prop2}, the average of $w$ over the intersection of any interval $J$ carrying $w_{n-l}$
with an interval not contained in $J$ is at most $\frac{9}{2}\cdot 2^l$.
\end{proof}

Recall that, given an interval $I$, we denoted by $I_m$ ($m=0,\dots,k-1$) the interval with the same
right endpoint as $I$ of length $|I_m|=\frac{1}{3^m}|I|$. These intervals have already appeared in the definition of $w_\nu(\o,\si,I)$.

\begin{prop}\label{partcase} Assume that $I$ carries $w_{n-l}$. Then
$$\int_{I_m}(\widetilde w)^2\si\le 30p^2w(I_m)\quad(l=0,\dots,n;\ m=0,\dots,k-2).$$
\end{prop}

\begin{proof} First, notice that the case when $l=n$ is trivial, since $2^{n+1}\le 4pw(x)$ on any interval $I$ carrying $w_0$,
and hence,
\begin{equation}\label{w0}
\int_{J}(\widetilde w)^2\si\le 16p^2w(J)\quad\text{for every }J\subset \operatorname{supp}w_0\,.
\end{equation}

Suppose now that $l\le n-1$ and consider first the case $m=0$. Assume that $I$ carries $w_{n-l}$.
For $j=0,\dots,n-l-1$ denote
$$F_j=I\cap\big(\text{supp}\,w_{n-(l+j)}\setminus \text{supp}\,w_{n-(l+j+1)}\big)$$
and let $E_j$ be the union of the tail intervals contained in~$F_j$.
Observe that $w=\frac{2^{l+j}}{p}$ on $F_j\setminus E_j$, and hence,
$\widetilde w(x)=2pw(x)$ for $x\in F_j\setminus E_j,$ which implies
\begin{equation}\label{part1}
\int_{\cup_j(F_j\setminus E_j)}(\widetilde w)^2\si\le 4p^2w(I).
\end{equation}

On the other hand, $w=\frac{4\e}{1+\e}\frac{2^{l+j}}{p}$ on $E_j$ and, as we have seen in the previous section,
$$|E_j|=\Big(\frac{1}{2}\Big(1-\frac{1}{3^{k-1}}\Big)\Big)^j\frac{1}{3^k}|I|\le \frac{1}{2^j}\frac{1}{3^k}|I|.$$
Combining this with Proposition \ref{prop1} yields
\begin{eqnarray}
\int_{\cup_jE_j}(\widetilde w)^2\si&\le& 4\sum_{j=0}^{n-l-1}2^{2(l+j)}\frac{(1+\e)p}{4\e 2^{l+j}}\frac{1}{2^j}\frac{1}{3^k}|I|\nonumber\\
&\le& \frac{2p}{\e}\frac{n}{3^k}2^l|I|\le 4p^22^l|I|=4p^2w(I).\label{part2}
\end{eqnarray}

Further, by (\ref{w0}),
$$
\int_{I\cap\text{supp}\,w_0}(\widetilde w)^2\si\le 16p^2w(I\cap\text{supp}\,w_0)\le 16p^2w(I).
$$
Combining this estimate with (\ref{part1}) and (\ref{part2}), we obtain
\begin{eqnarray}
\int_{I}(\widetilde w)^2\si&=&\int_{\cup_j(F_j\setminus E_j)}(\widetilde w)^2\si+\int_{\cup_jE_j}(\widetilde w)^2\si\label{m=0}\\
&+&\int_{I\cap\text{supp}\,w_0}(\widetilde w)^2\si\le 24p^2w(I),\nonumber
\end{eqnarray}
and this completes the proof in the case $m=0$.

Assume now that $1\le m\le k-2$. Notice that $I_m\setminus I_{m+1}=I_m^{(1)}\cup I_m^{(2)}$, where
$I_m^{(2)}$ carries $w_{n-(l+1)}$, and
$\widetilde w(x)=2pw(x)$ on $I_m^{(1)}$.
Thus, by (\ref{m=0}),
\begin{eqnarray}
\int_{I_m\setminus I_{m+1}}(\widetilde w)^2\si&\le& 4p^2w(I_m^{(1)})+24p^2w(I_m^{(2)})\label{it}\\
&\le& 24p^2w(I_m\setminus I_{m+1}).\nonumber
\end{eqnarray}

Further,
$$
\int_{I_{k-1}}(\widetilde w)^2\si\le 4(2^l)^2\frac{(1+\e)p}{4\e 2^l}|I_{k-1}|\le 6p2^l|I|.
$$
On the other hand, by Proposition \ref{prop1},
$$
w(I_m)\ge w(I_m^{(2)})=2^{l+1}|I_m^{(2)}|=\frac{2^{l+1}}{3^{m+1}}|I|,
$$
which, combined with the previous estimate, implies
$$
\int_{I_{k-1}}(\widetilde w)^2\si\le p3^{m+2}w(I_m)\le \frac{p}{\e}w(I_m)\le 6p^2w(I_m).
$$
Therefore, using (\ref{it}), we obtain
$$
\int_{I_m}(\widetilde w)^2\si= \sum_{j=m}^{k-2}\int_{I_j\setminus I_{j+1}}(\widetilde w)^2\si+\int_{I_{k-1}}(\widetilde w)^2\si\le 30p^2w(I_m),
$$
which completes the proof.
\end{proof}

We now turn to the proof of (\ref{sufshow}). Let $J\in {\mathcal J}$.
First consider the simple case when $|J|\ge 1$. In this case, $|J|=k$ for some $k\in {\mathbb N}$. Using that $w$ and $\widetilde w$ are $1$-periodic
along with the fact that $\int_0^1w=1$, and combining Propositions \ref{max} and \ref{partcase}, we obtain
$$
\frac{1}{w(J)}\int_J(M(w\chi_J))^2\si \le 25\int_0^1(\widetilde w)^2\si\le 25\cdot 30p^2.
$$

Suppose that $|J|< 1$. We can represent $J$ as the union of two triadic intervals $J=J_-\cup J_+$, where
$J_-,J_+\in {\mathcal T}$ are the left and the right halves of $J$ respectively. Since $J_-$ is triadic,
we must have that $|J_-|\le \frac{1}{3}$.
Also, by the $1$-periodicity of~$w$, one can assume that $J_{-}\subset [0,1)$.

Consider the case when $J$ contains an interval carrying $\text{supp}\,w_{n-(l+1)}$ for some $l$. Out of
all such intervals choose the longest one. Note that
since $|J|\le\frac 23$, we must have $l\ge 0$ in this case. Thus, the interval in question
must be of the kind $R^{(2)}_m$ where $R$ is an interval carrying $w_{n-l}$. Since neither $R=R_0$,
nor $R^{(2)}_{m-1}$ (if $m\ge 1$) is contained in $J$, there are only three possible options:

\begin{itemize}
\item $J_-=R_m^{(2)}$, $0\le m\le k-2$;
\item $J_+=R_m^{(2)}$, $0\le m\le k-2$;
\item $J_-=R_m$, $1\le m\le k-2$.
\end{itemize}

Suppose first that $J_-=R_m^{(2)}$ or $J_+=R_m^{(2)}$. Then $J\subset R_m$. By (\ref{intest}), $w(R_m)\le 3\cdot 2^{l-1}|R_m|$. On the other hand, since $R_m^{(2)}$ carries $w_{n-(l+1)}$,
by Proposition \ref{prop1},
\begin{equation}\label{someest}
w(J)\ge w(R_m^{(2)})=2^{l+1}|R_m^{(2)}|=2^{l+1}\frac{|R_m|}{3},
\end{equation}
which implies $w(R_m)\le \frac{9}{4}w(J)$. Therefore, by Propositions \ref{max} and~\ref{partcase},
$$\int_{J}(M(w\chi_{J}))^2\si\le 25\int_{R_m}(\widetilde w)^2\si\le 25\cdot 30p^2w(R_m)\le 25\cdot 75p^2w(J).$$

Assume now that $J_-=R_m, 1\le m\le k-2$. Then $w\equiv \frac{2^{l-1}}{p}$ on $J_+$ if $l>0$ and
$w\equiv \frac{2^{l}}{p}$ on $J_+$ if $l=0$. In either case, $\widetilde w=2pw$ on $J_+$, so
$$\int_{J_+}(\widetilde w)^2\si = 4p^2 w(J_+),$$
and thus, by Propositions \ref{max} and~\ref{partcase},
\begin{eqnarray*}
\int_{J}(M(w\chi_{J}))^2\si&\le& 25\Big(\int_{R_m}(\widetilde w)^2\si+\int_{J_+}(\widetilde w)^2\si\Big)\\
&\le& 25(30p^2w(R_m)+4p^2w(J_+))\le 25\cdot 30 p^2w(J).
\end{eqnarray*}

It remains to consider the case when $J$ does not contain an interval carrying $w_{n-(l+1)}$ for any $0\le l\le n-1$.
Denote by $E$ the union of all tail intervals appearing in the definition of $w$. Notice that if $x\not\in E$, then
$$
\sum_{l=1}^n2^l\chi_{\text{supp}\,w_{n-(l-1)}\setminus\text{supp}\,w_{n-l}}(x)=2pw(x)\chi_{{\mathbb R}\setminus \text{supp}\,w_0}.
$$
Also, $2^{n+1}\le 4pw(x)\chi_{\text{supp}\,w_0}$. From this and from Proposition \ref{max},
$$Mw(x)\le 18pw(x)\quad(x\not\in E)\,.$$
Therefore, if $J\cap E=\varnothing$,
\begin{equation}\label{notE}
\frac{1}{w(J)}\int_{J}(M(w\chi_{J}))^2\si\le 18^2p^2.
\end{equation}

Suppose that $J\cap E\ne\varnothing$. Then there exists $R$ carrying $w_{n-l}$ for some $0\le l\le n-1$ such that $J\cap R_{k-1}^{(3)}\ne\varnothing$.
Since $J_-,J_+$ and $R_{k-1}^{(3)}$
are triadic, we have that either one half of $J$ is contained in $R_{k-1}^{(3)}$ or $R_{k-1}^{(3)}\subset J$. Since $J$ cannot contain any interval carrying $\text{supp}\,w_{n-(l+1)}$, in both cases we obtain that
$w$ can take only three possible values
$$\frac{2^l}{p},\frac{4\e 2^l}{(1+\e)p},\frac{2^{l-1}}{p}$$
on $J$
and therefore,
$$
\frac{1}{w(J)}\int_{J}(M(w\chi_{J}))^2\si\le \big(\sup_{J}w/\inf_{J}w\big)^2\le
\Big(\frac{1+\e}{4\e}\Big)^2\le 9p^2,
$$
which along with (\ref{notE}) implies
$$
\frac{1}{w(J)}\int_{J}(M(w\chi_{J}))^2\si\le 18^2 p^2.
$$
This completes the proof of (\ref{sufshow}), and therefore the first estimate in (\ref{mainpr}) is proved.

\subsection{Estimate of the Hilbert transform}
The goal of this section is to prove the second estimate in (\ref{mainpr}).

Denote by $A_l^*,l=0,\dots,n-1,$ the union of all intervals $\frac{1}{2}I$ where $I$ is a tail interval
contained in
$$[0,1)\cap\big(\text{supp}\,w_{n-l}\setminus \text{supp}\,w_{n-(l+1)}\big).$$
In other words, $A_l^*$ is the union of all intervals $\frac 12 J_{k-1}^{(3)}$ where $J\subset [0,1)$ carries $w_{n-l}$.
Then, by (\ref{Al}),
\begin{equation}\label{Al*}
|A_l^*|=\frac{1}{2}|A_l|=\frac{1}{2}\left(\frac{1}{2}\Big(1-\frac{1}{3^{k-1}}\Big)\right)^l\frac{1}{3^k}\quad(l=0,\dots,n-1).
\end{equation}
The sets $A_l^*$ plays the central role in establishing the lower bound for $H(w\chi_{[0,1)})$, as the following proposition shows.

\begin{prop}\label{hilbest}
There exists an absolute $C>0$ such that for for all $l=0,\dots,n-1$ and for every $x\in A_l^*$
\begin{equation}\label{below}
|H(w\chi_{[0,1)})(x)|\ge Ck2^l.
\end{equation}
\end{prop}

Let us first show how to derive the second estimate in (\ref{mainpr}) from here. By (\ref{Al*}) and (\ref{below}),
$$
\int_{A_l^*}|H(w\chi_{[0,1)})|^2\si\ge C^2k^22^{2l}\frac{1+\e}{4\e}\frac{p}{2^l}\frac{1}{2}\left(\frac{1}{2}\Big(1-\frac{1}{3^{k-1}}\Big)\right)^l\frac{1}{3^k}.
$$
Therefore,
\begin{eqnarray*}
&&\|H(w\chi_{[0,1)})\|^2_{L^2(\si)}\ge \sum_{l=0}^{n-1}\int_{A_l^*}|H(w\chi_{[0,1)})|^2\si\\
&&\ge \frac{C^2}{8}k^2p\sum_{l=0}^{n-1}\Big(1-\frac{1}{3^{k-1}}\Big)^l=\frac{C^2}{24}k^23^kp\left(1-\Big(1-\frac{1}{3^{k-1}}\Big)^n\right).
\end{eqnarray*}
Since  $n=3^{k-1}$ and $(1-1/n)^{n}<1/{\rm{e}}$, we obtain
$$\|H(w\chi_{[0,1)})\|^2_{L^2(\si)}\ge\frac{C^2(1-1/{\rm e})}{24}k^23^kp
\ge \frac{C^2(1-1/{\rm e})}{144(\log 3)^2}t^2\log^2t.$$

Let us now turn to the proof of Proposition \ref{hilbest}. Let $J=[a,b)\subset[0,1)$ be an interval carrying
$w_{n-l}$. Assume that $x\in \frac{1}{2}J_{k-1}^{(3)}$.
Write
\begin{eqnarray*}
&&H(w\chi_{[0,1)})(x)=H(w\chi_{[0,a)})(x)+\sum_{m=0}^{k-2}H(w\chi_{J_m\setminus J_{m+1}})(x)\\
&&+H(w\chi_{J_{k-1}\setminus J_{k-1}^{(3)}})(x)+H(w\chi_{J_{k-1}^{(3)}})(x)+H(w\chi_{[b,1)})(x)\\
&&\equiv A(x)+B(x)+C(x)+D(x)+E(x).
\end{eqnarray*}

We will show that there are
absolute constants $C_1$ and $C_2$ such that for all $x\in \frac{1}{2}J_{k-1}^{(3)}$,
\begin{equation}\label{both}
|B(x)|\ge C_1k2^l\quad\text{and}\quad\max\{|D(x)|,|E(x)|\}\le C_22^l.
\end{equation}
Since $A(x),B(x)$ and $C(x)$ are positive for all $x\in \frac{1}{2}J_{k-1}^{(3)}$, we obtain from (\ref{both}) that
\begin{eqnarray*}
|H(w\chi_{[0,1)})(x)|&\ge& |A(x)+B(x)+C(x)|-|D(x)|-|E(x)|\\
&\ge& |B(x)|-|D(x)|-|E(x)|\ge \frac{C_1}{2}k2^l
\end{eqnarray*}
for $k>\frac{4C_2}{C_1}$.

Now let us prove the first estimate in (\ref{both}). If $y\in J_m\setminus J_{m+1}$ and $x\in \frac{1}{2}J_{k-1}^{(3)}$, then
$0\le x-y\le |J_m|$. Using also that $J_m^{(2)}\subset J_m\setminus J_{m+1}$, by Proposition \ref{prop1} we obtain
\begin{eqnarray*}
H(w\chi_{J_m\setminus J_{m+1}})(x)&=&\int_{J_m\setminus J_{m+1}}\frac{w(y)}{x-y}dy\ge \frac{w(J_{m}\setminus J_{m+1})}{|J_m|}\\
&\ge&\frac{w(J_m^{(2)})}{|J_m|}=\frac{2}{3}2^l.
\end{eqnarray*}
Therefore,
$$B(x)>\frac{2}{3}(k-1)2^l.$$

Turn to the second part of (\ref{both}). Let $J_{k-1}^{(3)}=[\a,b)$. Then for all $x\in \frac{1}{2}J_{k-1}^{(3)}$,
$$
\left|\int_{J_{k-1}^{(3)}}\frac{w(y)}{x-y}dy\right|=\frac{4\e}{(1+\e)}\frac{2^l}{p}\left|\log\Big|\frac{x-\a}{x-b}\Big|\right|\le 4(\log 3)\e\frac{2^l}{p}\le 4(\log 3)2^l.
$$

It remains to estimate $|E(x)|$. Take the intervals $J^i=[a_i,b_i),i=0,\dots, l$ such that $J^i$ carries $w_{n-l+i}$ and
$$J=J^0\subset J^1\subset\dots\subset J^l=[0,1).$$

We claim that for every $i=1,\dots,l$ and for all $x$ such that $0<x\le b_{i-1}-\frac{|J^{i-1}|}{4\cdot 3^{k}}$,
\begin{equation}\label{claim}
|H(w\chi_{[b_{i-1},b_i)})(x)|\le 13\cdot 2^{l-i}.
\end{equation}
Notice first that this claim immediately implies the desired estimate for $E(x)$. Indeed, let $x\in \frac{1}{2}J_{k-1}^{(3)}$. Then $0<x\le b-\frac{|J|}{4\cdot 3^{k}}$,
and hence (\ref{claim}) holds for $i=1$. But since $x\not\in (J^i)_{k-1}$ for all $i=1,\dots, l$, we obviously obtain that $0<x\le b_{i-1}-\frac{|J^{i-1}|}{4\cdot 3^{k}}$
for all $i\le l$. Therefore, by (\ref{claim}),
$$
|H(w\chi_{[b,1)})(x)|\le \sum_{i=1}^l|H(w\chi_{[b_{i-1},b_i)})(x)|\le 13\sum_{i=1}^l2^{l-i}\le 13\cdot 2^l.
$$

It remains to prove the claim. Denote $x_i=b_{i-1}-\frac{|J^{i-1}|}{4\cdot 3^{k}}$. Observe that
$|H(w\chi_{[b_{i-1},b_i)})(x)|$ is an increasing function for $x<b_{i-1}$. Therefore, it suffices to prove that
\begin{equation}\label{sufic}
|H(w\chi_{[b_{i-1},b_i)})(x_i)|\le 13\cdot 2^{l-i}.
\end{equation}

There exists $0\le m\le k-2$ such that $J^{i-1}=(J_m^i)^{(2)}$. Then $[b_{i-1}, b_i)=J^i_{m+1}$. Let $h=|J^i_{m+1}|$.
Split the integral in (\ref{sufic}) as follows:
$$
\int_{b_{i-1}}^{b_i}\frac{w(y)}{y-x_i}dy=\int_{b_{i-1}}^{b_{i-1}+h/3}\frac{w(y)}{y-x_i}dy+\int_{b_{i-1}+h/3}^{b_i}\frac{w(y)}{y-x_i}dy.
$$
Using that $w\equiv \frac{2^{l-i}}{p}$ on $[b_{i-1}, b_{i-1}+h/3)$, we obtain
$$
\int_{b_{i-1}}^{b_{i-1}+h/3}\frac{w(y)}{y-x_i}dy\le \frac{2^{l-i}}{p}\frac{h}{3}\frac{4\cdot 3^{k}}{h}\le 8\cdot 2^{l-i}.
$$

Next, applying (\ref{intest}) yields
$$
\int_{b_{i-1}+h/3}^{b_i}\frac{w(y)}{y-x_i}dy\le \frac{3}{h}\frac{3}{2}2^{l-i}|J_{m+1}^i|=\frac{9}{2}\cdot 2^{l-j},
$$
which along with the previous estimate proves (\ref{sufic}).

This proves the claim and so Proposition \ref{hilbest}. Thus, Theorem \ref{main} is completely proved.

\section{Appendix}
In this section, we will show how to prove (\ref{keytool}).
Let us show first that for every interval $I\subset {\mathbb R}$, there exists an interval $J\in {\mathcal J}$ containing $I$ and such that $|J|\le 6|I|$.
Indeed, let $I=[a,a+h)$. Fix $j\in {\mathbb Z}$ such that $3^{j-1}\le h<3^j$, and take $n\in {\mathbb Z}$ such that
$$3^jn\le a<3^j(n+1).$$
Then $I\subset J=[3^jn,3^j(n+2))$,
and $\frac{|J|}{|I|}\le \frac{2\cdot 3^{j}}{3^{j-1}}=6$.

It follows immediately from this property that
\begin{equation}\label{triadic}
Mf(x)\le 6M^{\mathcal J}f(x),
\end{equation}
where
$$M^{\mathcal J}f(x)=\sup_{J\ni x, J\in {\mathcal J}}\frac{1}{|J|}\int_J|f|dy.$$

Next, it is easy to see that the intervals from ${\mathcal J}$ can be split into two disjoint triadic lattices, ${\mathcal J}={\mathcal T}^{1}\cup{\mathcal T}^2$
(see Fig. \ref{twotriadic} for a geometric illustration of this fact).

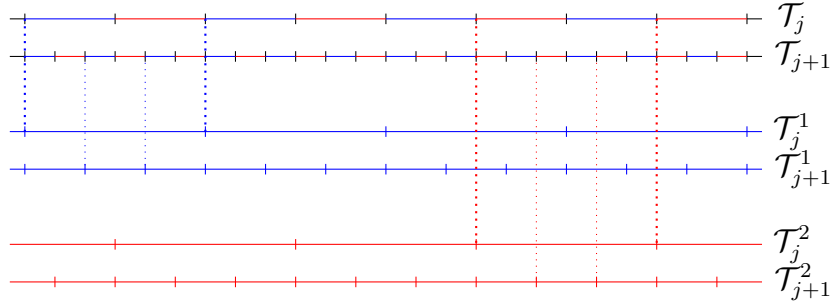
\begin{figure}[h]
\begin{centering}
\begin{tikzpicture}

\draw (10.4,0) node {${\mathcal T}_j$};
\draw(0,0)--(0.2,0);
\draw(9.8,0)--(10,0);
\foreach \x/\xtext in {0.2/$$,1.4/$$,2.6/$$,3.8/$$,5/$$,6.2/$$,7.4/$$,8.6/$$,9.8/$$}
   \draw(\x,2pt)--(\x,-2pt) node[below] {\xtext};

\draw[blue] (0.2,0) -- (1.4,0);
\draw[red]  (1.4,0) -- (2.6,0);
\draw[blue] (2.6,0) -- (3.8,0);
\draw[red]  (3.8,0) -- (5,0);
\draw[blue] (5,0) -- (6.2,0);
\draw[red]  (6.2,0) -- (7.4,0);
\draw[blue] (7.4,0) -- (8.6,0);
\draw[red]  (8.6,0) -- (9.8,0);

\begin{scope}[yshift=-0.5cm]
\draw (10.55, 0) node {${\mathcal T}_{j+1}$};
\draw(0,0)--(0.2,0);
\draw(9.8,0)--(10,0);
\foreach \x/\xtext in {0.2/$$,0.6/$$,1/$$,1.4/$$,1.8/$$,2.2/$$,2.6/$$,3/$$,3.4/$$,3.8/$$,4.2/$$,4.6/$$,5/$$,5.4/$$,5.8/$$,6.2/$$,6.6/$$,7/$$,7.4/$$,7.8/$$,8.2/$$,8.6/$$,9/$$,9.4/$$,9.8/$$}
   \draw(\x,2pt)--(\x,-2pt) node[below] {\xtext};

\draw[blue] (0.2,0) -- (0.6,0);
\draw[red]  (0.6,0) -- (1,0);
\draw[blue] (1,0) -- (1.4,0);
\draw[red]   (1.4,0) -- (1.8,0);
\draw[blue] (1.8,0) -- (2.2,0);
\draw[red] (2.2,0) -- (2.6,0);
\draw[blue] (2.6,0) -- (3,0);
\draw[red] (3,0) -- (3.4,0);
\draw[blue] (3.4,0) -- (3.8,0);
\draw[red] (3.8,0) -- (4.2,0);
\draw[blue] (4.2,0) -- (4.6,0);
\draw[red] (4.6,0) -- (5,0);
\draw[blue] (5,0) -- (5.4,0);
\draw[red] (5.4,0) -- (5.8,0);
\draw[blue] (5.8,0) -- (6.2,0);
\draw[red] (6.2,0) -- (6.6,0);
\draw[blue] (6.6,0) -- (7,0);
\draw[red]  (7,0) -- (7.4,0);
\draw[blue] (7.4,0) -- (7.8,0);
\draw[red] (7.8,0) -- (8.2,0);
\draw[blue] (8.2,0) -- (8.6,0);
\draw[red] (8.6,0) -- (9,0);
\draw[blue] (9,0) -- (9.4,0);
\draw[red] (9.4,0) -- (9.8,0);
\end{scope}

\begin{scope}[yshift=-1.5cm]
\draw (10.4,0) node {${\mathcal T}^1_j$};

\draw[blue](0,0)--(10,0);
\foreach \x/\xtext in {0.2/$$,2.6/$$,5/$$,7.4/$$,9.8/$$}
   \draw[blue](\x,2pt)--(\x,-2pt) node[below] {\xtext};
\end{scope}

\begin{scope}[yshift=-2cm]
\draw (10.55, 0) node {${\mathcal T}^1_{j+1}$};
\draw[blue](0,0)--(10,0);
\foreach \x/\xtext in {0.2/$$,1/$$,1.8/$$,2.6/$$,3.4/$$,4.2/$$,5/$$,5.8/$$,6.6/$$,7.4/$$,8.2/$$,9/$$,9.8/$$}
   \draw[blue](\x,2pt)--(\x,-2pt) node[below] {\xtext};
\end{scope}

\begin{scope}[yshift=-3cm]
\draw (10.4,0) node {${\mathcal T}^2_j$};
\draw[red](0,0)--(10,0);
\foreach \x/\xtext in {1.4/$$,3.8/$$,6.2/$$,8.6/$$}
   \draw[red](\x,2pt)--(\x,-2pt) node[below] {\xtext};
\end{scope}

\begin{scope}[yshift=-3.5cm]
\draw (10.55, 0) node {${\mathcal T}^2_{j+1}$};
\draw[red](0,0)--(10,0);
\foreach \x/\xtext in {0.6/$$,1.4/$$,2.2/$$,3/$$,3.8/$$,4.6/$$,5.4/$$,6.2/$$,7/$$,7.8/$$,8.6/$$,9.4/$$}
   \draw[red](\x,2pt)--(\x,-2pt) node[below] {\xtext};
\end{scope}

\draw[blue,thick,dotted] (0.2,0)--(0.2,-1.5);
\draw[blue,thick,dotted] (2.6,0)--(2.6,-1.5);
\draw[blue,dotted] (1,-0.5)--(1,-2);
\draw[blue,dotted] (1.8,-0.5)--(1.8,-2);
\draw[red,thick,dotted] (6.2,0)--(6.2,-3);
\draw[red,thick,dotted] (8.6,0)--(8.6,-3);
\draw[red,dotted] (7,-0.5)--(7,-3.5);
\draw[red,dotted] (7.8,-0.5)--(7.8,-3.5);

\end{tikzpicture}
\caption{The lattices ${\mathcal T}, {\mathcal T}^1$ and ${\mathcal T}^2$ shown at two consecutive generations.
The unions of blue and adjacent (from the right) red intervals from ${\mathcal T}_j$ form ${\mathcal T}_j^1$,
and the unions of red and adjacent (from the right) blue intervals from ${\mathcal T}_j$ form ${\mathcal T}_j^2$.
In turn, the unions of blue and red children from ${\mathcal T}_{j+1}$ form ${\mathcal T}_{j+1}^1$, and the unions
of red and blue children from ${\mathcal T}_{j+1}$ form ${\mathcal T}_{j+1}^2$.}
\label{twotriadic}
\end{centering}
\end{figure}

Therefore, by (\ref{triadic}),
\begin{equation}\label{tr1}
\|M\|_{L^2(\si)\to L^2(\si)}\le 6\Big(\|M^{{\mathcal T}^1}\|_{L^2(\si)\to L^2(\si)}+\|M^{{\mathcal T}^2}\|_{L^2(\si)\to L^2(\si)}\Big).
\end{equation}

In order to estimate the right-hand side of (\ref{tr1}), we invoke the following proposition.

\begin{prop}\label{sawyer} Let ${\mathfrak T}$ be a triadic lattice. Then
$$\|M^{\mathfrak T}\|_{L^2(\si)\to L^2(\si)}\le 2\sup_{R\in {\mathfrak T}}\left(\frac{1}{w(R)}\int_R(M^{\mathfrak T}(w\chi_R))^2\si\right)^{1/2}.$$
\end{prop}

\begin{remark}
For dyadic lattices this result can be found in \cite{M}. The proof there is closely related to the approach by E. Sawyer \cite{S}
in his two-weighted characterization for the maximal operator. For triadic lattices the proof is essentially the same, and we give it
for the sake of completeness.
\end{remark}


\begin{proof}[Proof of Proposition \ref{sawyer}]
Let $a>1$. For $k\in {\mathbb Z}$ write the set $\O_k=\{M^{\mathfrak T}f>a^k\}$
as the union of pairwise disjoint triadic intervals $I_j^k$ such that $\frac{1}{|I_j^k|}\int_{I_j^k}|f|>a^k$.
Denote $E_j^k=I_j^k\setminus \O_{k+1}$, and set $\a_{j,k}=(w(I_j^k)/|I_j^k|)^2\si(E_j^k).$
We have
\begin{eqnarray}
\|M^{\mathfrak T}f\|_{L^2(\si)}^2&=&\sum_{k\in {\mathbb Z}}\int_{\Omega_k\setminus \Omega_{k+1}}(M^{\mathfrak T}f)^2\si\le a^2\sum_{k,j}\left(\frac{1}{|I_j^k|}\int_{I_j^k}|f|\right)^2\si(E_j^k)\nonumber\\
&=&a^2\sum_{k,j}\left(\frac{1}{w(I_j^k)}\int_{I_j^k}|f\si|w\right)^2\a_{j,k}\,.\label{firstpart}
\end{eqnarray}

Notice that for every $R\in {\mathfrak T}$,
\begin{equation}\label{carlprop}
\sum_{j,k:I_j^k\subset R}\a_{j,k}\le \int_R(M^{\mathfrak T}(w\chi_R))^2\si\le N^2w(R),
\end{equation}
where
$$N=\sup_{R\in {\mathfrak T}}\left(\frac{1}{w(R)}\int_J(M^{\mathfrak T}(w\chi_R))^2\si\right)^{1/2}.$$

For $\la>0$ set
$$E_{\la}=\left\{(j,k):\left(\frac{1}{w(I_j^k)}\int_{I_j^k}|f\si|w\right)^2>\la\right\}.$$

Define the weighted maximal operator $M_w^{\mathfrak T}$ by
$$
M_w^{\mathfrak T}f(x)=\sup_{J\ni x, J\in \mathfrak T}\frac 1{w(J)}\int_J |f|w\,dy\,.
$$
Writing the set $\{x:M_w^{\mathfrak T}(f\si)^2(x)>\la\}$ as the union of the maximal pairwise disjoint triadic intervals $\cup_iR_i$
and applying (\ref{carlprop}), we obtain
\begin{eqnarray*}
\sum_{(j,k)\in E_{\la}}\a_{j,k}\le \sum_i\sum_{j,k:I_j^k\subset R_i}\a_{j,k}\le N^2w\{x:M_w^{\mathfrak T}(f\si)^2(x)>\la\}.
\end{eqnarray*}
Therefore,
\begin{eqnarray*}
&&\sum_{k,j}\left(\frac{1}{w(I_j^k)}\int_{I_j^k}|f\si|w\right)^2\a_{j,k}=\int_0^{\infty}\Big(\sum_{(j,k)\in E_{\la}}\a_{j,k}\Big)d\lambda\\
&&\le N^2\|M_w^{\mathfrak T}(f\si)\|_{L^2(w)}^2\le 4N^2\|f\si\|_{L^2(w)}^2 = 4N^2\|f\|_{L^2(\si)}^2,
\end{eqnarray*}
which, along with (\ref{firstpart}), completes the proof.
\end{proof}

\end{document}